\documentclass[12pt]{article}

\usepackage{amssymb, amsmath, fullpage}

\newtheorem{definition}{Definition}
\newtheorem{proposition}{Proposition}
\newtheorem{theorem}{Theorem}
\newtheorem{example}{Example}

\newtheorem{corollary}{Corollary}

\newcommand{\gf}{\mathop{\mathrm{GF}}}
\newcommand{\rk}{\mathop{\mathrm{rk}}}
\newcommand{\cl}{\mathop{\mathrm{cl}}\nolimits}

\newcommand{\pbd}{\mathop{\mathrm{PBD}}}
\newcommand{\mr}{\mathop{\mathrm{mr}}}
\newcommand{\Mr}{\mathop{\mathrm{Mr}}}

\newenvironment{proof}{\textbf{Proof}}{$\Box$\\}

\begin{document}
\title{Combinatorial Representations}
\author{Peter J. Cameron, Maximilien Gadouleau and S\o ren Riis\\
School of Mathematical Sciences (PJC)\\
School of Electronic Engineering and Computer Science (MG, SR)\\
Queen Mary, University of London\\
Mile End Road\\
London E1 4NS, UK}

\maketitle

\begin{abstract}
This paper introduces combinatorial representations, which generalise the notion of linear representations of matroids. We show that any family of subsets of the same cardinality has a combinatorial representation via matrices. We then prove that any graph is representable over all alphabets of size larger than some number depending on the graph. We also provide a characterisation of families representable over a given alphabet. Then, we associate a rank function and a rank operator to any representation which help us determine some criteria for the functions used in a representation. While linearly representable matroids can be viewed as having representations via matrices with only one row, we conclude this paper by an investigation of representations via matrices with only two rows.
\end{abstract}


\section{Definition and examples} \label{sec:definition}

Combinatorial representations, defined below, generalise the notion of (linear) representations
of matroids.

\begin{definition} \label{def:comb_rep}
Let $E$ be a set of $n$ elements
and $\mathcal{B}$ a family of $r$-element subsets of $E$. %(Think of $\mathcal{B}$ as the set of bases of a matroid of rank~$r$ on $E$.)

A \emph{combinatorial representation} of $(E,\mathcal{B})$ over a set $X$
is defined to be an $n$-tuple of functions $f_i:X^r\to X$ such that,
for any $r$ distinct indices $i_1,\ldots,i_r\in E$, the map from
$X^r$ to $X^r$ given by
\[
    (x_1,x_2,\ldots,x_r)\mapsto (f_{i_1}(x_1,\ldots,x_r),\ldots, f_{i_r}(x_1,\ldots,x_r))
\]
is bijective if and only if $\{i_1,\ldots,i_r\}\in\mathcal{B}$.
\end{definition}

Usually we assume $E=\{1,2,\ldots,n\}$.
%If $e \in E$ does not appear in any basis in $\mathcal{B}$, then the choice for $f_e$ is rather uninteresting, e.g. any constant function would work. Therefore, we assume $\bigcup_{b \in \mathcal{B}} b = E$ and we only consider $\mathcal{B}$ unless otherwise specified.
We denote the map given by the displayed equation by $f_b$, where $b = \{i_1,\ldots,i_r\}$. This slight abuse of notation will not be detrimental to the rest of the paper. Remark that the cases where $r=1$ or $r=n$ are trivial.

\begin{example} \label{ex:B_22}
Let $n=4$ and $\mathcal{B}=\{\{1,2\},\{3,4\}\}$. A
combinatorial representation over a $3$-element set $\{a,b,c\}$ is
given by taking $f_1$ and $f_2$ to be the two coordinate functions
(that is, $f_1(x,y)=x$ and $f_2(x,y)=y$), and $f_3$ and $f_4$ by the
tables
\[
\begin{array}{|c|c|c|}\hline
b&a&a\\\hline b&c&b\\\hline c&c&a\\\hline
\end{array}\hbox{ and }
\begin{array}{|c|c|c|}\hline
b&b&c\\\hline a&c&c\\\hline a&b&a\\\hline
\end{array}.
\]
Note that $(E,\mathcal{B})$ is not a matroid.
\end{example}

\paragraph{Remark 1.} Suppose that $b = \{i_1,\ldots,i_r\}\in\mathcal{B}$.
Define functions $g_i$, for $i\in E$, by
\[
    g_i(x_1,\ldots,x_r)=f_i(y_1,\ldots,y_r),
\]
where $(y_1,\ldots,y_r)$ is the inverse image of $(x_1,\ldots,x_r)$
under the bijection $f_b$. These functions also
define a combinatorial representation, with the property that $g_{i_j}$
is the $j$th coordinate function. So, where necessary, we may suppose
that the first $r$ elements of $E$ form a basis and the first $r$
functions are the coordinate functions. This transformation can be viewed as a change of variables.

\paragraph{Remark 2.} The values of the functions $f_i$ are not significant; the definition
could be written in terms of the partitions of $A^r$ given by these functions:
$\pi_i=\{\{x\in A^r:f_i(x)=a\}:a\in A\}$. Thus, we require that the meet
(in the partition lattice) of $r$ partitions is the partition into singletons
if and only if the indices of these partitions form a set in $\mathcal{B}$.

\paragraph{Remark 3.} The condition that the domain of the functions is $A^r$ is also
not essential; any set of cardinality $q^r$ will do (where $q=|A|$), since
as in Remark 1 the functions corresponding to a set in $\mathcal{B}$ give
this set the structure of a Cartesian power.

\paragraph{Remark 4.} Our point of view is similar to that of experimental design in
statistics, where functions on (or partitions of) the set of experimental
units are called \emph{factors}, see~\cite[Chapter 10]{Bai08}.

To take a very simple example, let us assume that $q^2$ trees in an orchard
are laid out in a $q\times q$ square. Last year, $q$ fertilizers were applied
to the trees, using a Latin square layout, so that each fertilizer was used
once in each row and column. This year, we want to test $q$ pesticides on
the trees, again in a Latin square layout; but, because of possible interaction
between fertilizer and pesticide,  we would like each combination to occur
just once. We can regard rows, columns, fertilizers and pesticides as four
functions from the set of trees to a set of size $q$ (or four factors, each
with $q$ parts of size $q$); our requirement is that we have a representation
of the complete graph of size~$4$.

We return to this in Example~2.

\paragraph{Problem} Suppose that a set family $(E,\mathcal{B})$ has a
combinatorial representation. Is there a simple condition which
guarantees that $(E,\mathcal{B})$ is a matroid?

\begin{theorem} \label{th:linear}
A set family is a linearly representable matroid if and only if it has a combinatorial representation by linear functions.
\end{theorem}

\begin{proof}
A linear representation of
a matroid in an $r$-dimensional vector space over a field $F$ associates a column vector (the $i$th column of the
$r\times n$ matrix over $F$) with each element of $E$ so that a
set of columns is linearly independent if and only if the corresponding
set of elements of $E$ is independent in the matroid. If we associate with
each column vector $(a_1,\ldots,a_r)^\top$ the function from $F^r$ to $F$
given by $(x_1,\ldots,x_r)\mapsto\sum a_ix_i$, we see that the functions
defined by the columns of the matrix satisfy the requirements of the
definition of a combinatorial representation.

Conversely, a set family represented by linear functions over a
field $F$ must be a matroid. To show this, we must verify the exchange
axiom. Let $B_1,B_2\in\mathcal{B}$; as in the preceding remark, we may
assume that the elements of $B_1$ are the coordinate functions. Now consider the $r-1$ functions $f_i$ for $i\in B_2$, $i\ne k$, for some
fixed $k\in B_2$. These define a surjective function from $F^r$ to
$F^{r-1}$. Take any non-zero vector in the kernel, and suppose that its
$l$th coordinate is non-zero. Then it is readily checked that
$B_2\setminus\{k\}\cup\{l\}$ is a basis.
\end{proof}

\begin{example}\label{ex:U2n} A combinatorial representation of the uniform matroid
$U_{2,n}$ on a set of size~$r$ is equivalent to a set of $n-2$
mutually orthogonal Latin squares of order~$r$.

For suppose that $f_1$ and $f_2$ are the two coordinate functions. Then
the maps $(i,j)\mapsto(i,f_k(i,j))$ and $(i,j)\mapsto(f_k(i,j),j)$ are
bijective; so $f_k$ is a Latin square for all $k$. Also, the map
$(i,j)\mapsto(f_k(i,j),f_l(i,j))$ is bijective for $l\ne k$; so
the Latin squares $f_k$ and $f_l$ are orthogonal. The argument reverses.
\end{example}

More generally, a representation of $U_{r,n}$ over an alphabet $A$ is equivalent to an $(n,r,n-r+1)$ MDS code over $A$.

\section{All set families are representable} \label{sec:all}

In this section, we show that any family is representable over some finite alphabet by giving an explicit construction via matrix linear functions. The proof is based on first representing the uniform matroid which contains all $r$-element subsets, and then removing subsets via cartesian products of representations.

\begin{proposition}
Let $(E,\mathcal{B}_1)$ and $(E,\mathcal{B}_2)$ be families of $r$-sets,
which have representations over alphabets of cardinalities $q_1$ and
$q_2$ respectively. Then $(E,\mathcal{B}_1\cap\mathcal{B}_2)$ has a
representation over an alphabet of size $q_1q_2$.
\label{prop:cartesian}
\end{proposition}

\begin{proof}
Suppose that $(f_e)$ and $(g_e)$ are representations of $(E,\mathcal{B}_1)$ and
$(E,\mathcal{B}_2)$ over alphabets $A_1$ and $A_2$ respectively. Consider the functions $h_e:(A_1 \times A_2)^r\to A_1 \times A_2$ given by
\[
    h_e((a_1,b_1),\ldots,(a_r,b_r))=(f_e(a_1,\ldots,a_r),g_e(b_1,\ldots,b_r)).
\]
It is tedious but routine to show that, for any $b \subseteq E$, $h_b$
is a bijection if and only if both $f_b$ and
$g_b$ are bijections. So the functions $(h_e:e\in E)$
represent $(E,\mathcal{B}_1\cap\mathcal{B}_2)$.
\end{proof}

We can now prove the theorem.

\begin{theorem} \label{th:all}
Any family is representable over some finite alphabet by matrix linear functions.
\end{theorem}

\begin{proof}
First of all, if $\mathcal{B} = U_{r,n}$, then it has a representation by linear functions. Otherwise, we can express $\mathcal{B}$ as
$$
    \mathcal{B} = \bigcap_{c \in U_{r,n} \backslash \mathcal{B}} U_{r,n} \backslash \{c\}.
$$

We now give a linear representation of $U_{r,n} \backslash \{c\}$. Without loss, let us assume $c = \{1,\ldots,r\}$. For any prime power $p \geq {n-1 \choose r-1} + 1$, there are $n-1$ vectors $v_2,\ldots,v_n$ such that $v_{i_1},\ldots,v_{i_r}$ are linearly independent for any choice of indices; moreover, there is $v_1 \in \langle v_2,\ldots,v_r \rangle$ such that $v_1 \notin \langle v_{i_1},\ldots,v_{i_{r-1}} \rangle$ for any other choice of indices, since
$$
    \left| \langle v_2,\ldots,v_r \rangle \backslash \bigcup_{i_1,\ldots,i_{r-1}} \langle v_{i_1},\ldots,v_{i_{r-1}} \rangle \right| \geq p^{r-1} - {n-1 \choose r-1}p^{r-2} > 0.
$$
These vectors thus form a linear representation.

By applying the cartesian product construction in Proposition \ref{prop:cartesian}, we obtain a matrix linear representation of $(E,\mathcal{B})$.
\end{proof}

%\textbf{Comments} Clearly, the representation used in the proof of Theorem \ref{th:all} may not be a representation over the smallest possible alphabet. Can we find a smaller value for $p$ in the proof above? %What conclusion can we draw about infinite alphabets?

In order to illustrate our concepts and results, we consider the family of bases $(E,\mathcal{B}_{r,k})$ for any integers $r,k$, where $E = \{0,\ldots,rk-1\}$ and
$$
    \mathcal{B}_{r,k} = \{\{0,\ldots,r-1\}, \{r,\ldots,2r-1\},\ldots, \{(k-1)r,\ldots,kr-1\}\}.
$$

The case $k=r=2$ has been already studied in Example \ref{ex:B_22}, so we assume $r \geq 3$ or $k \geq 3$ henceforth. We first give a combinatorial representation of $\mathcal{B}_{r,k}$ over $\mathbb{Z}_k$. Remark that the functions used in that representation are not matrix linear.

\begin{proposition} \label{prop:Brn}
Let $D = \{x \in \mathbb{Z}_k^r : x_0 = \ldots = x_{r-1}\}$ be the diagonal and let $\chi_D$ be its characteristic function. Then the functions
$$
    f_{rm+s}(x) = x_s + m \chi_D(x)
$$
for $0 \leq s \leq r-1$, $0 \leq m \leq k-1$ form a combinatorial representation of $\mathcal{B}_{r,k}$ over $\mathbb{Z}_k$.
\end{proposition}

\begin{proof}
First, we prove that for each $b_m = \{mr,\ldots, (m+1)r-1\} \in \mathcal{B}$, $f_{b_m}$ is a permutation of $\mathbb{Z}_k^r$. Remark that $f_{b_m}(x) \in D$ if and only if $x \in D$, then it is easily shown that $f_{b_m}$ is a bijection on $D$ and also on $\mathbb{Z}_k^r \backslash D$. Second, we prove that any $f_b = (f_{m_1r + s_1}, \ldots, f_{m_r r + s_r})$ where $b \notin \mathcal{B}$ does not form a permutation of $\mathbb{Z}_n^r$. If $s_i = s_j$ for some $i \neq j$ (without loss, $s_i=s_j=0$), then if $k \geq 3$
$$
    (f_{m_ir},f_{m_jr})(0,1,0,\ldots,0) = 0 = (f_{m_ir},f_{m_jr})(0,2,0,\ldots,0),
$$
and if $k=2$, $r \geq 3$,
$$
    (f_{m_ir},f_{m_jr})(0,1,0,\ldots,0) = 0 = (f_{m_ir},f_{m_jr})(0,0,1,\ldots,0).
$$
Therefore, if $s_i = s_j$, $f_b$ is not a permutation. We now turn to the case where all $s_i$'s are distinct (without loss $s_i = i$), then there exist $i \neq j$ with $m_i \neq m_j$ and
$$
    f_b(0,\ldots,0) = (m_1,\ldots,m_r) = f_b(m_1,\ldots,m_r).
$$
\end{proof}

\section{Representations of graphs} \label{sec:graphs}

Representability is not a monotonic property of alphabet size. For example,
a representation of the complete graph on $4$ vertices is equivalent to a
pair of orthogonal Latin squares; these exist over alphabets of sizes  $3$,
$4$ and $5$ but not $6$. However, we will prove the following.

\begin{theorem} \label{th:graph}
Let $(E,\mathcal{B})$ be a graph (a family of sets of cardinality $2$). Then
$(E,\mathcal{B})$ has combinatorial representations over all sufficiently
large finite alphabets.
\label{t1}
\end{theorem}

The theorem follows from the two propositions below. %I will give the proof of the theorem, and then turn to the proof of the propositions.

Consider set systems with $r=2$, that is, graphs. We say that a representation
$(f_e:e\in E)$ of the graph $(E,\mathcal{B})$ is \emph{idempotent} if
$f_e(x,x)=x$ for all $x\in A$, where $A$ is the alphabet.

\begin{proposition} \label{prop:idempotent}
Let $(E,\mathcal{B}_1)$ and $(E,\mathcal{B}_2)$ be graphs,
which have idempotent representations over alphabets of cardinalities $q_1$ and
$q_2$ respectively. Then $(E,\mathcal{B}_1\cap\mathcal{B}_2)$ has an idempotent
representation over an alphabet of size $q_1q_2$.
\label{p2}
\end{proposition}

The proof is the same as the one for Proposition \ref{prop:cartesian} and hence omitted.

For the second proposition, we need to recall the terminology of Richard
Wilson~\cite{Wil74}. A \emph{pairwise balanced design} consists of a set $X$ and a
family $\mathcal{L}$ of subsets of $X$ with the property that any two
distinct elements of $X$ are contained in a unique member of $\mathcal{L}$.
It is a $\pbd(K)$, where $K$ is a set of positive integers, if the cardinality
of every member of $\mathcal{L}$ is contained in $K$. The elements of $X$ and
$\mathcal{L}$ are called \emph{points} and \emph{lines} respectively.

A set $K$ of positive integers is \emph{PBD-closed} if, whenever there exists a
$\pbd(K)$ with $v$ points, then $v\in K$. Given any set $K$ of positive
integers, define
\begin{eqnarray*}
    \alpha(K) &=& \gcd\{k-1:k\in K\}, \\
    \beta(K) &=& \gcd\{k(k-1):k\in K\}.
\end{eqnarray*}
Wilson's main theorem asserts that a PBD-closed set $K$ contains all but
finitely many integers $v$ such that $\alpha(K)\mid v-1$ and
$\beta(K)\mid v(v-1)$.

\begin{proposition} \label{prop:PBD-closed}
Let $(E,\mathcal{B})$ be a graph. Then the set of cardinalities of alphabets
over which $(E,\mathcal{B})$ has an idempotent combinatorial representation is
PBD-closed.
\label{p3}
\end{proposition}

\begin{proof}
Let $K$ be the set of alphabet sizes for which the graph $G=(E,\mathcal{B})$
has an idempotent representation. To show that $K$ is PBD-closed, let
$(X,\mathcal{L})$ be a $\pbd(K)$ on $v$ points; we have to show that $v\in  K$.

By assumption, for each line $L$ of the PBD, we have an idempotent
representation $(f_e^L)$ with the alphabet $L$. We construct a representation
$(f_e)$ with alphabet $X$ by the following rule:
\begin{itemize}
\item $f_e(x,x)=x$;
\item if $x\ne y$, and $L$ is the unique line containing $x$ and $y$, then
$f_e(x,y)=f_e^L(x,y)$.
\end{itemize}
We claim that this is a representation.

Take $e_1\ne e_2$. Suppose first that $(f_{e_1},f_{e_2})$ is not a bijection.
Then there exist distinct pairs $(x,y)$ and $(x',y')$ such that
$f_{e_i}(x,y)=f_{e_i}(x',y')$ for $i=1,2$. We consider three cases.
\begin{itemize}
\item If $x=y$ and $x'=y'$, then $x=f_{e_i}(x,y)=f_{e_i}(x',y')=x'$.
\item Suppose that $x=y$ and $x'\ne y'$. Then $f_{e_i}(x',y')=f_{e_i}(x,y)=x$,
so $x$ lies in the line $L$ containing $x'$ and $y'$. Then
$(f_{e_1}^L,f_{e_2}^L)$ is not a bijection, so $\{e_1,e_2\}\notin\mathcal{B}$.
\item Suppose that $x\ne y$ and $x'\ne y'$. If $f_{e_1}(x,y)\ne f_{e_2}(x,y)$,
then both of these points lie in the line $L$ containing $x$ and $y$; hence
$x'$ and $y'$ also lie
in this line. Now $(f_{e_1}^L,f_{e_2}^L)$ fails to be a bijection, and so
$\{e_1,e_2\}\notin\mathcal{B}$. So we can suppose that
$f_{e_1}(x,y)=f_{e_2}(x,y)=z$, say, with $x,y,z\in L$. Then
\[(f_{e_1}(x,y),f_{e_2}(x,y))=(f_{e_1}(z,z),f_{e_2}(z,z)),\]
and again $(f_{e_1}^L,f_{e_2}^L)$ fails to be a bijection.
\end{itemize}

Conversely, suppose that $\{e_1,e_2\}\notin\mathcal{B}$. Then, for any
line $L$,  $(f_{e_1}^L,f_{e_2}^L)$ is not a bijection; so $(f_{e_1},f_{e_2})$
is not a bijection.

Our claim is proved, and with it, the Proposition.
\end{proof}

\paragraph{Proof of the Theorem} First, we observe that the complete graph
$K_n$ has an idempotent representation over any field with at least $n$
elements: simply associate a field element $\lambda(e)$ with each $e\in E$,
and put $f_e(x,y)=\lambda(e)x+(1-\lambda(e))y$.

Now we obtain an idempotent representation of the complete graph minus an edge:
if  $e_1$ and $e_2$ are the two nonadjacent vertices, take the above
representation of the graph on $E\setminus\{e_1\}$, and let $f_{e_2}=f_{e_1}$.

Next, an arbitrary graph is the intersection of graphs each of which is a
complete graph minus an edge, and so has an idempotent representation, by
Proposition~\ref{p2}. If the alphabet size of this representation is $N$, we
obtain further representations over alphabets of size $qN$, for any $q$ whose
prime factors are each at least the number of vertices (by Proposition~\ref{p2}
again, intersecting with a complete graph).

Now to prove the Theorem, we know from Proposition~\ref{p3} that the set $K$ of
alphabet sizes over which idempotent representations exist is PBD-closed; so by
Wilson's theorem we have only to show that $\alpha(K)=1$ and $\beta(K)=2$.

Suppose that $\alpha(K)>1$. Then every number of the form $qN$ as above is
congruent to $1$ mod~$\alpha(K)$, contradicting the fact that we can choose
an arbitrarily large prime $p\equiv-1$~(mod~$\alpha(K))$. So $\alpha(K)=1$.
The argument for $\beta(K)$ is similar. $\Box$

\section{Families representable over a given finite alphabet} \label{sec:given_alphabet}

\subsection{Characterisation}

We now characterize families which are representable over a given alphabet. Clearly, if $(E,\mathcal{B})$ is representable over a finite alphabet $A$, then it is representable over any other alphabet with the same cardinality, so we assume $A = \mathbb{Z}_q$ unless otherwise specified. Furthermore, if $(E,\mathcal{B})$ is representable over $A$, then any induced subgraph of $(E,\mathcal{B})$ is as well.

First of all, the definitions below easily generalize concepts for matroids. For any $e \in E$, denote the set of bases containing $e$ (taking $e$ away) as $\mathcal{B}(e) = \{b \subseteq E \backslash \{e\}: |b| = r-1, b \cup \{e\} \in \mathcal{B}\}$.

\begin{definition} \label{def:loop}
\begin{itemize}
    \item $l$ is a \emph{loop} if no basis contains $l$, that is, $\mathcal{B}(l)=\emptyset$.

    \item $l_1$ and $l_2$ are \emph{parallel} if each can be replaced by the other in a basis, that is, $\mathcal{B}(l_1)=\mathcal{B}(l_2)$.

    \item The subset $I$ of $E$ is \emph{dependent} if no basis contains $I$.
\end{itemize}
\end{definition}

These definitions are absolute, that is, independent of representation. However,
given a representation $f=(f_i:i\in E)$, we can say:
\begin{itemize}
    \item $l$ is an \emph{$f$-loop} if the partition corresponding to $l$ does not have all its parts of the same size.

    \item $l_1$ and $l_2$ are \emph{$f$-parallel} if the corresponding partitions are equal. (A statistician would say that these factors are \emph{aliased}.)

    \item A subset $I$ of $E$, with $|I|=s$, is \emph{$f$-dependent} if the meet of the partitions indexed by $I$ in the partition lattice does not have $q^s$ parts of size $q^{r-s}$.
\end{itemize}
The representation-specific notions just defined imply the absolute notions
given earlier. In the case of a linear representation of a matroid, the
concepts are equivalent.

We say that a set $I$ is \emph{$f$-independent} if it is not $f$-dependent.

Remark: if $\{l_1,l_2\}$ is $f$-independent, then the partitions
corresponding to $l_1$ and $l_2$ are orthogonal in the statistical sense.
The converse is false, but in fact $l_1$ and $l_2$ are independent if and
only if the partitions are orthogonal and their join is the partition with
a single part. See~\cite[Chapter 10]{Bai08} for definitions.

It is clear that the parallel relation is an equivalence relation. An obvious choice can be made to represent a family with loops and parallel elements: use constant functions for loops, and use the same function for any set of parallel elements. We clarify this idea below.

\begin{definition} \label{def:balanced}
A function $f:A^r \rightarrow A$ is {\em balanced} if for all $a \in A$, $|f^{-1}(a)| = q^{r-1}$. Two balanced functions $f,g : A^r \rightarrow A$ are {\em parallel} if there exists a permutation $\pi \in S_A$ such that $f = \pi \circ g$.
\end{definition}

It is easily proved that $f$ is balanced if and only if there exists a permutation $\sigma \in S_{A^r}$ such that $f = g \circ \sigma$, where $g$ is any coordinate function.

Again, the parallel relation for functions is an equivalence relation, where each equivalence class contains $q!$ functions. Indeed, any balanced function can be viewed as a partition of $A^r$ into $q^{r-1}$ parts of $q$ elements each (the set of pre-images $f^{-1}(a)$ for all $a$). Two functions are parallel if and only if they induce the same partition, and the equivalence class can be viewed as that partition, which we shall denote as $\bar{f}$.

\begin{proposition} \label{prop:simple}
Let $(E,\mathcal{B})$ be a family of bases of rank $r$ and $f_i$ be a representation of $(E,\mathcal{B})$ over $A$. Then
\begin{enumerate}
    \item \label{item:loop} If $l \in E$ is not a loop, then $f_l$ is balanced. Otherwise, $f_l$ can be chosen to be any imbalanced function.

    \item \label{item:parallel} If $l$ and $m$ are not parallel, then $f_l$ and $f_m$ are not parallel either. Otherwise, $f_l$ and $f_m$ can be chosen to be parallel.
\end{enumerate}
\end{proposition}

\begin{proof}
\ref{item:loop}) If $l$ is not a loop, then there exists $b \in \mathcal{B}$ which contains $l$, say $b = \{l_1 = l,l_2,\ldots,l_r\}$. Therefore, $|f_b^{-1}(x)| = 1$ for all $x \in A^r$ which implies $|f_l^{-1}(x_1)| = q^{r-1}$ for all $x_1 \in A$ and hence $f_l$ is balanced.\\
\\
\ref{item:parallel}) If $f_l$ and $f_m$ are parallel and $f_l$ can be extended to a bijection by $f_2,\ldots,f_r$, then we easily obtain that $f_m$ can also be extended by the same functions. Therefore, $l$ and $m$ are parallel.
\end{proof}

We refer to a family of bases as {\em simple} if it contains neither loops nor parallel elements. Any family of bases $(E,\mathcal{B})$ can be turned into a simple one $(E^*,\mathcal{B}^*)$ by removing loops and considering one element per parallel class. Proposition \ref{prop:simple} then indicates that $(E,\mathcal{B})$ is representable over an alphabet $A$ if and only if $(E^*,\mathcal{B}^*)$ is representable over the same alphabet.

By Proposition \ref{prop:simple} above, all functions in the representation of a simple family are balanced and non-parallel to one another. As a corollary, they are all distinct, which shows that there are only finitely many representable simple families of a given rank and over a given alphabet. We can now characterise these families.

\begin{definition} \label{def:Mqr}
Let $P(q,r)$ be the set of partitions of $A^r$ into $q$ equal parts and denote its elements as $\bar{f}_1,\ldots,\bar{f}_k$. Let $\mathcal{M}(q,r) = (P(q,r), \mathcal{B})$, where $\mathcal{B} = \{\{\bar{f}_{i_1},\ldots,\bar{f}_{i_r}\} : \bar{f}_{\{i_1,\ldots,i_r\}} \mbox{has}\, q^r \, \mbox{parts}\}$.
\end{definition}

Clearly, $\mathcal{M}(q,r)$ is simple and representable over $A$: to any partition $\bar{f}_i$ associate the corresponding function $f_i$. Therefore, all its induced subgraphs are also representable over $A$ (but may not be simple).  Moreover, all representable families `belong to' $\mathcal{M}(q,r)$.

\begin{theorem} \label{th:subgraph}
The family $(E,\mathcal{B})$ of rank $r$ is representable over $A$ if and only if $(E^*,\mathcal{B}^*)$ is isomorphic to an induced subgraph of $\mathcal{M}(q,r)$.
\end{theorem}

\begin{proof}
First, as mentioned above, $(E,\mathcal{B})$ is representable over $A$ if and only if $(E^*, \mathcal{B}^*)$ is representable over $A$. The latter is equivalent to the existence of functions $(f_e)$ such that $\{i_1,\ldots,i_r\} \in \mathcal{B}^*$ if and only if $f_{\{i_1,\ldots,i_r\}}$ is a bijection, which holds if and only if $\{\bar{f}_{i_1},\ldots,\bar{f}_{i_r}\}$ is a basis of $\mathcal{M}(q,r)$.
\end{proof}

\subsection{Properties of $\mathcal{M}(q,r)$}

Proposition \ref{prop:M(q,r)} below enumerates some properties of $\mathcal{M}(q,r)$.

\begin{proposition} \label{prop:M(q,r)}
The hypergraph $\mathcal{M}(q,r)$ satisfies the following properties.
\begin{enumerate}
    \item The number of vertices is given by
    $$
        |P(q,r)| = \frac{1}{q!} {q^r \choose q^{r-1},\ldots,q^{r-1}} = \frac{q^r!}{q!(q^{r-1}!)^q}.
    $$

    \item $\mathcal{M}(q,r)$ is regular in the following sense. For any set of $1 \leq k \leq r$ partitions $\{\bar{f}_1,\ldots,\bar{f}_k\}$ which are in a basis of $\mathcal{M}(q,r)$, there are $N(q,r;k,l)$ sets of $l$ partitions $\{\bar{f}_{k+1},\ldots,\bar{f}_{k+l}\}$ such that $\bar{f}_1,\ldots,\bar{f}_{k+l}$ belong to a basis, where
    $$
        N(q,r;k,l) = \frac{1}{l!(q!)^l} {q^{r-k} \choose q^{r-k-l},\ldots,q^{r-k-l}}^{q^k} = \frac{(q^{r-k}!)^{q^k}}{q!^l(q^{r-k-l}!)^{q^{k+l}}}.
    $$

    \item  In particular, $\mathcal{M}(q,r)$ is regular of valency
    $$
        N(q,r;1,r-1) = \frac{(q^{r-1}!)^q}{(r-1)!(q!)^{r-1}}.
    $$
    Therefore, the number of bases in $\mathcal{M}(q,r)$ is given by
    $$
        |\mathcal{B}| = \frac{q^r!}{r! (q!)^r}.
    $$

    \item $\mathcal{M}(q,r)$ contains $\mathcal{M}(q,r-1)$ in the following sense. Let $g_0(x,x_r) = x_r$ for all $x = (x_1,\ldots,x_{r-1})$; for any partition $\bar{f} \in P(q,r-1)$, let $\bar{g} \in P(q,r)$ be defined as $g(x,x_r) = f(x)$. Then $\{\bar{f}_1,\ldots,\bar{f}_{r-1}\} \in \mathcal{M}(q,r-1)$ if and only if $\{\bar{g}_0, \bar{g}_1,\ldots,\bar{g}_{r-1}\} \in \mathcal{M}(q,r)$.
\end{enumerate}
\end{proposition}

\begin{proof}
1. The number of balanced functions of $A^r$ to $A$ is exactly the multinomial coefficient ${q^r \choose q^{r-1},\ldots,q^{r-1}}$. Since any balanced function has exactly $q!$ parallel functions, we obtain the value of $|P(q,r)|$.

2. Let us denote the function generated by $f_1,\ldots,f_k$ as $f$ and the one generated by $f_{k+1},\ldots,f_{k+l}$ as $g$. The set $\{\bar{f}_1,\ldots,\bar{f}_{k+l}\}$ belongs to a basis if and only if the function $h:A^r \to A^{k+l}$ is balanced. In other words, $g$ must be balanced over all pre-images $f^{-1}(a)$ and hence can be viewed as $q^k$ functions $g_a:A^{r-k} \to A^l$. There are exactly ${q^{r-k} \choose q^{r-k-l},\ldots,q^{r-k-l}}$ choices for each $g_a$, and hence ${q^{r-k} \choose q^{r-k-l},\ldots,q^{r-k-l}}^{q^k}$ choices for $g$. Accounting for all parallel functions and all permutations of $\{\bar{f}_{k+1},\ldots,\bar{f}_{k+l}\}$, we must divide by $l!(q!)^l$ to obtain the value of $N(q,r;k,l)$.

3. $N(q,r;1,r-1)$ is a special case of the property above, while $|\mathcal{B}|$ is easily obtained by double counting.

4. This is clear by definition of $\mathcal{M}(q,r)$.
\end{proof}

This gives criteria on representability of families.

\begin{corollary} \label{cor:representable}
Let $(E,\mathcal{B})$ be a simple family of rank $r$. If $|E| > |P(q,r)|$ or if there exists a set $X$ of $k$ elements such that there are more than $N(q,r;k,l)$ sets $Y$ of $l$ elements with $X \cup Y$ are in a basis of $\mathcal{B}$, then $(E,\mathcal{B})$ is not representable over any alphabet of size up to $q$.
\end{corollary}

%\begin{conjecture}
%Make the following graph from $\mathcal{M}(q,kr)$: its vertex set is all the $k$-subsets of $P(q,kr)$ which appear in some edge of $\mathcal{M}(q,kr)$; there is an edge of size $r$ among $b_1,\ldots,b_r$ if and only if $b_1 \cup \ldots \cup b_r$ forms an edge in $\mathcal{M}(q,kr)$. Then this graph is isomorphic to $\mathcal{M}(q^k,r)$.
%\end{conjecture}

We remark that the clique number of $\mathcal{M}(q,r)$ is of particular interest, as a clique of size $n$ corresponds to an $(n,r)$ MDS code over an alphabet of size $q$. Therefore, the MDS conjecture~\cite[Research Problem 11.4]{MS77} can be recast a conjecture of the clique number of $\mathcal{M}(q,r)$.

%\subsection{Adjacency of partitions}

Proposition \ref{prop:adjacency} below gives necessary and sufficient conditions for adjacency in the graph $\mathcal{M}(q,r)$. For any two functions $f,g : A^r \to A$, we define the Hamming distance between $f$ and $g$ as
$$
    d_H(f,g) := |\{x \in A^r : f(x) \neq g(x)\}|.
$$

\begin{proposition} \label{prop:adjacency}
Let $\bar{f}, \bar{g} \in P(q,r)$, then the following are equivalent.
\begin{enumerate}
    \item \label{it:edge} They are adjacent in $\mathcal{M}(q,r)$, i.e. $\{\bar{f}, \bar{g}\} \subseteq b$ for some $b \in \mathcal{M}(q,r)$.

    \item \label{it:balanced_edge} The function $g$ restricted to the set $f^{-1}(a)$ is balanced for all $a \in A$.

    \item \label{it:parallel_edge} For all $f'$ parallel to $f$, $d_H(f',g) = (q-1)q^{r-1}$.

    \item \label{it:subset_permutations_edge} There exists $T \subseteq S_q$ such that $|T| = (q-1)^2$, there exist $a,b$ with $a \pi \neq b$ for all $\pi \in T$, and all permutation matrices $\{M_\pi \in \mathbb{R}^{q \times q}: \pi \in T\}$ are linearly independent for which $d(\pi \circ f, g) = (q-1)q^{r-1}$ for all $\pi \in T$.
\end{enumerate}
\end{proposition}

\begin{proof}
The first two properties are clearly equivalent (see the proof of Proposition \ref{prop:M(q,r)}). Let us prove that the second one implies the third one. Suppose $g$ is balanced over all pre-images of $f$ (and hence, over all pre-images of $f'$ for any parallel $f'$ of $f$). Then $g$ agrees with $f'$ in $q^{r-2}$ positions on each pre-image; there are $q$ pre-images, yielding $d_H(f',g) = q^r - q^{r-1}$. Also, the third property clearly implies the fourth one.

Let us now show that the fourth property implies the second one. Foremost, recall that the subspace $P$ of $\mathbb{R}^{q \times q}$ spanned by all $q \times q$ permutation matrices has dimension $(q-1)^2 + 1$ \cite{FM60}. For all $0 \leq i \leq q-1$, consider the matrix $R_i \in \mathbb{R}^{q \times q}$ whose only nonzero row $i$ is equal to the all-zero vector; consider also $C_j = R_j^T$ for all $j$. Together, these $2q$ matrices span a linear subspace $CR$ of dimension $2q-1$ (the all-ones matrix ${\bf 1}^{q \times q}$ is in the intersection). Let us prove that $\dim(CR \cap P) = 1$ (again, the line spanned by ${\bf 1}^{q \times q}$). Suppose $N \in CR \cap P$, then
$$
    N = \sum_\pi \gamma_\pi M_\pi = \sum_i \alpha_i R_i + \sum_j \beta_j C_j.
$$
The sum of all entries on each row and each column is equal to $s = \sum_\pi \gamma_\pi$. For row $i$, this yields $s = \sum_j \beta_j + n\alpha_i$, from which we have $\alpha_i = \alpha$ is a constant; similarly, we have $\beta_j = \beta$ and $N = (\alpha+\beta){\bf 1}^{q \times q}$.

Since the $(a,b)$ entry of $M_\pi$ is zero for all $\pi \in T$, the subspace $Q$ spanned by the permutation matrices in $T$ does not contain the all-ones matrix and hence $Q \oplus CR = \mathbb{R}^{q \times q}$. We now expand those matrices to $q^2$-dimensional row vectors by concatenating their rows. We can represent $g$ via a column vector $\gamma \in \mathbb{R}^{q^2}$ where $\gamma_{i+qj} = |f^{-1}(i) \cap g^{-1}(j)|$. The fact that $g$ is balanced is equivalent to $C_j \gamma = q^{r-1}$ for all $j$; similarly, $f$ balanced yields $R_i \gamma = q^{r-1}$ for all $0 \leq i \leq q-2$. Also, $d(\pi \circ f, g) = q^r - q^{r-1}$ yields $M_\pi \gamma = q^{r-1}$ for all $\pi \in T$. Overall, we obtain $M\gamma = q^{r-1}{\bf 1}^{q^2 \times 1}$, where $M \in \mathbb{R}^{q^2 \times q^2}$ is non-singular. Thus there is a unique solution, given by $\gamma = q^{r-2}{\bf 1}^{q^2 \times 1}$, i.e. when $g$ restricted to the set $f^{-1}(a)$ is balanced for all $a \in A$.
\end{proof}

We remark that Condition \ref{it:parallel_edge} could be replaced by: $d_H(f',g)$ is a constant for all $f'$ parallel to $f$. This holds since double counting yields
$$
    \sum_{\pi \in S_A} d_H(\pi \circ f,g) = q^r (q-1)(q-1)!
$$
for any function $g: A^r \rightarrow A$ (not necessarily balanced). However, this simplification cannot be applied to Condition \ref{it:subset_permutations_edge}.

%We conclude this section by realizing that since we know by Theorem \ref{th:all} that any family can be represented by matrix linear functions, we could similarly define a graph which contains all non-parallel balanced matrix linear functions. The vertices in the graph can then be viewed as matrices

\section{Rank and closure} \label{sec:rk_cl}

\subsection{Rank}

This section generalises some concepts from matroid theory to the idea of combinatorial representations. Let us first define a rank function.

\begin{definition} \label{def:rank}
Let $(E,\mathcal{B})$ be a family of bases of rank $r$. A function $\rk : 2^E \rightarrow [0,r]$ is a {\em rank function} for $(E,\mathcal{B})$ if it satisfies the following conditions:
\begin{itemize}
    \item If $X \subseteq E$, then $0 \leq \rk(X) \leq |X|$.

    \item If $X \subseteq Y \subseteq E$, then $\rk(X) \leq \rk(Y)$.

    \item $\rk$ is submodular, i.e. if $X$ and $Y$ are subsets of $E$, then
    \begin{equation} \label{eq:submodular}
        \rk(X \cup Y) + \rk(X \cap Y) \leq \rk(X) + \rk(Y).
    \end{equation}

    \item If $|X| = r$, then $\rk(X) = r$ if and only if $X \in \mathcal{B}$.
\end{itemize}
\end{definition}

The first three conditions are equivalent to: $\rk(X)$ is a polymatroid~\cite{Oxl06}. This definition implies that if $X \subseteq b \in \mathcal{B}$, then $|X| = \rk(X)$.

It is easy to show that the only case where modular equality is reached for all subsets is for the uniform matroid $U_{r,r}$. Indeed, we have for any $e \in E$
$$
    \rk(E \backslash \{e\}) = \rk(E) + \rk(\emptyset) - \rk(\{e\}) = r-1,
$$
and hence any basis must contain $e$. Thus only $E$ itself can be a basis, and $(E,\mathcal{B}) = U_{r,r}$.

Below we show that every representation leads to a rank function. In particular, if $f$ is a representation by matrix-linear functions, then its corresponding rank function takes rational values.

\begin{proposition} \label{prop:r_f}
Let $(E,\mathcal{B})$ be a family of bases and let $f = (f_1,\ldots,f_n)$ be a representation for it over a finite alphabet of size $q$. Then $r_f$ defined as
$$
    r_f(X) := H(f_X) = - \sum_{a \in f_X(A^r)} \frac{|f_X^{-1}(a)|}{q^r} \log_q \left\{ \frac{|f_X^{-1}(a)|}{q^r} \right\} = r - q^{-r} \sum_{a \in f_X(A^r)} |f_X^{-1}(a)| \log_q |f_X^{-1}(a)|
$$
where $H$ is the $q$-ary entropy function, is a rank function for $(E,\mathcal{B})$.
\end{proposition}

\begin{proof}
The proof of submodularity simply follows Shannon's inequality~\cite[Eq. (2.93)]{CT91} and was already given in~\cite{ZY97}. The other properties are straightforward.
\end{proof}

\paragraph{Remark 1.} A subset $I$ of $E$ is $f$-independent if and only if $r_f(I)=|I|$.

\paragraph{Remark 2.} The converse of Proposition \ref{prop:r_f} is not true: there exist rank functions which do not correspond to any combinatorial representation. This fact was proved in~\cite{ZY97}, where they showed that any entropy function satisfies an additional inequality. In particular, they demonstrate that the rank function used in the proof of Theorem \ref{th:half_integers} below with $p=2$ over the family $(E = \{1,2,3,4\},\mathcal{B} = \{\{1,2\}\})$ cannot be viewed as the entropy function of any representation for $(E,\mathcal{B})$. %Since then, many other non-Shannon inequalities have been discovered for the entropy.

Let us denote
$$
    \mr(X) = \max_{b \in \mathcal{B}}|b \cap X|,
$$
when $(E,\mathcal{B})$ is a matroid, this is its rank function. The rank function for the uniform matroid $U_{r,n}$ is given by
$$
    \Mr(X) = \min\{r,|X|\}.
$$
It is easily shown that for any family $(E,\mathcal{B})$ and any rank function, we have
$$
    \mr(X) \leq \rk(X) \leq \Mr(X).
$$
Therefore, $(E,\mathcal{B})$ is a matroid if and only if it has an {\em integer-valued} rank function. Theorem \ref{th:half_integers} below shows that for any divisor $p \geq 2$, there exists a rank function which takes values over the integers divided by $p$. Therefore, although matroids are viewed as special due to the submodularity of the rank function, it seems that the particularity of matroids actually resides in the {\em integrality} of the rank function.

\begin{theorem} \label{th:half_integers}
For any $p \geq 2$, $(E,\mathcal{B})$ has a rank function which takes values over the integers divided by $p$.
\end{theorem}

\begin{proof}
We claim that the function $\rk(X)$ defined by
\begin{equation} \nonumber
    \rk(X) = \begin{cases}
    |X| & \mathrm{if}\, |X| \leq r-1 \,\mathrm{or}\, X \in \mathcal{B}\\
    r- \frac{1}{p} & \mathrm{if}\, |X| = r, X \notin \mathcal{B}\\
    r & \mathrm{if}\, |X| \geq r+1
    \end{cases}
\end{equation}
is a rank function for $(E, \mathcal{B})$. Only the submodular inequality is in question, so we have to check that (\ref{eq:submodular}) holds for any $X,Y \subseteq E$ (with $\rk(X) \geq \rk(Y)$ without loss of generality). This clearly holds (with equality) if $Y \subseteq X$ hence we assume this does not occur. We need to study four cases.
\begin{itemize}
    \item Case I: $|X|, |Y| \leq r-1$. Then
    $$
        \rk(X) + \rk(Y) = |X| + |Y| = |X \cup Y| + |X \cap Y| \geq \rk(X \cup Y) + \rk(X \cap Y).
    $$

    \item Case II: $\rk(X) = r - \frac{1}{p}$, $|Y| \leq r-1$. Then $\rk(X \cap Y) = |X \cap Y| \leq |Y| - 1$ and hence
    $$
        \rk(X) + \rk(Y) \geq r + (|Y| - 1) \geq \rk(X \cup Y) + \rk(X \cap Y).
    $$

    \item Case III: $\rk(X) = \rk(Y) = r - \frac{1}{p}$. Then $|X \cap Y| \leq r - 1$ and hence
    $$
        \rk(X) + \rk(Y) = r + (r - \frac{2}{p}) \geq \rk(X \cup Y) + \rk(X \cap Y).
    $$

    \item Case IV: $\rk(X) = r$. Then
    $$
        \rk(X) + \rk(Y) = r + \rk(Y) \geq \rk(X \cup Y) + \rk(X \cap Y).
    $$
\end{itemize}
\end{proof}

Theorem \ref{th:half_integers} also shows that the supremum $\Mr(X)$ of all rank functions can be approached. Therefore, one can only expect to derive lower bounds on any rank function, but not any upper bound other than $\Mr(X)$.

It seems difficult to design a generic rank function which takes significantly lower values than the one in Theorem \ref{th:half_integers}. However, by an argument similar to that of Theorem \ref{th:half_integers}, we can show that the following is a rank function for $(E,\mathcal{B})$:
$$
    \rk(X) = \begin{cases}
    |X| & \mathrm{if}\, X \subseteq b \in \mathcal{B}\\
    r & \mathrm{if}\, |X| \geq r+1\\
    |X| - 2^{|X| - r - 1} & \mathrm{otherwise}.
    \end{cases}
$$

\subsection{Bounds on rank functions}

In view of Proposition \ref{prop:r_f}, studying rank functions in general can help determine some constraints on the functions used in a representation. The lower bound of $\mr(X)$ for any rank function is tight for the rank function of a matroid. However, the exchange axiom implies that matroids are typically dense, i.e. the number of bases is large, and (leaving out trivial cases) for any basis $b$, there exists another basis $b'$ with $|b \cap b'| = r-1$. Therefore, in order to obtain lower bounds on the rank function which differ from $\mr(X)$, we consider sparse families.

Proposition \ref{prop:bound_I} shows that a single ``isolated basis'' leads to significant gap between $\mr(X)$ and the rank of $X$ for some $X$. The restriction on fractions only serves the sakes of conciseness and ease of presentation.

\begin{proposition} \label{prop:bound_I}
Let $(E,\mathcal{B})$ be a family of rank $r$, and let
$$
    I = \min_{b \in \mathcal{B}} \max_{c \in \mathcal{B}, c \neq b} |b \cap c|.
$$
Then there exists $X \subseteq E$ such that $|X| = r$, $\mr(X) = \frac{r+I}{2}$, and for any rank function $\rk(X)$ for $\mathcal{B}$, we have $\rk(X) \geq \frac{3r+I}{4}$. Thus
\begin{equation} \nonumber
    \rk(X) - \mr(X) \geq \frac{r-I}{4},
\end{equation}
provided all fractions are integers.
\end{proposition}

\begin{proof}
Let $b \in \mathcal{B}$ such that the intersection of any basis with $b$ has size at most $I$ and let $c \in \mathcal{B}$ such that $|b \cap c| = I$. Define two sets $X,Y \subseteq E$ as follows: $X = (b \cap c) \cup X' \cup Z$ and $Y = (b \cap c) \cup Y' \cup Z$, where $X',Y' \subseteq b$ and $Z \subseteq c$ have cardinality $J = \frac{r-I}{2}$ and none of the constituents intersect. Therefore, $|X| = |Y| = r$ and
\begin{equation} \nonumber
    \rk(X) + \rk(Y) \geq \rk(X \cup Y) + \rk(X \cap Y) = r + I + J,
\end{equation}
since $b \subseteq X \cup Y$ and $X \cap Y = Z \cup (b \cap c) \subseteq c$. Without loss, suppose $\rk(X) \geq \rk(Y)$, then $\rk(X) \geq \frac{r + I + J}{2} = \frac{3r+I}{4}$. Let us now show that $\mr(X) = I + J = \frac{r+I}{2}$. First, we have $|b \cap X| = |c \cap X| = I + J$. Second, for any other $d \in \mathcal{B}$, we have the following inequalities
\begin{eqnarray*}
%    |d \cap X'| &\leq& J\\
    |d \cap Z| \leq |Z| &\leq& J,\\
    |d \cap X'| + |d \cap (b \cap c)| \leq |d \cap b| &\leq& I,
\end{eqnarray*}
and hence
\begin{equation} \nonumber
    |d \cap X| = |d \cap (b \cap c)| + |d \cap X'| + |d \cap Z| \leq J + I.
\end{equation}
\end{proof}

The argument is strengthened in Theorem \ref{th:no_strong} below, which exhibits a set with $\mr(X) = 1$ and yet for which the rank is arbitrarily close to $r$. Recall that a {\em transversal} for $(E, \mathcal{B})$ is a set of elements in $E$ such that any $b \in \mathcal{B}$ contains at least one element of the transversal.

\begin{theorem} \label{th:no_strong}
Let $k$ denote the minimum size of a transversal for $(E,\mathcal{B})$, then there exists a set of elements $X \subseteq E$ such that $|X| = k$, $\mr(X) = 1$, and
\begin{equation} \nonumber
    \rk(X) \geq r \left( 1 - \left( 1- \frac{1}{r} \right)^k \right).
\end{equation}
\end{theorem}

\begin{proof}
We prove, by induction on $1 \leq i \leq k$, that there exists a set $X_i = \{e_1,\ldots,e_i\}$ with $\mr(X_i) = 1$ and $\rk(X_i) \geq r_i = r \left( 1 - \left( 1- \frac{1}{r} \right)^i \right)$. This trivially holds for $i=1$, let us assume this holds for $i-1$. Let $b = \{e_i^1,\ldots,e_i^r\}$ be a basis which does not intersect with $\{e_1,\ldots,e_{i-1}\}$. Note that such a basis exists, as $\{e_1,\ldots,e_{i-1}\}$ is not a transversal by definition of $k$. We have
$$
    \sum_{j=1}^r \rk(X_{i-1} \cup \{e_i^j\}) \geq \rk(X_{i-1} \cup b) + (r-1)\rk(X_{i-1}) \geq r + (r-1)r_{i-1},
$$
and hence $\rk(X_{i-1} \cup \{e_i^j\}) \geq 1 + (1 - 1/r)r_{i-1} = r_i$ for some $j$.
\end{proof}

\subsection{Closure}

We can generalise the idea of parallelism to the concept of closure, defined below. The point is that the closure operator depends on the rank function, not only on the family of bases.

\begin{definition} \label{def:closure}
For a rank function $\rk(X)$ of $(E,\mathcal{B})$, the closure associated to $\rk(X)$ of $X \subseteq E$ is given by
$$
    \cl_{\rk}(X) = \{e \in E : \rk(X \cup \{e\}) = \rk(X)\}.
$$
\end{definition}

If the rank function is known, we simply write $\cl(X)$.

\begin{proposition} \label{prop:closure}
The closure satisfies the following properties.
\begin{enumerate}
%    \item \label{it:cl(0)} For all $e \in E$, $\cl(\{e\}) = \{e\}$. Also, $\cl(\emptyset) = \emptyset$.

    \item \label{it:Xincl} For all $X \subseteq E$, $X \subseteq \cl(X)$.

    \item \label{it:cl_increasing} If $X \subseteq Y$, then $\cl(X) \subseteq \cl(Y)$.

    \item \label{it:cl(cl)} For all $X \subseteq E$, $\cl(\cl(X)) = \cl(X)$.

    \item \label{it:rk(cl)} For all $X \subseteq E$, $\rk(\cl(X)) = \rk(X)$.

    \item \label{it:cl=E} For all $X \subseteq E$, $\cl(X) = E$ if and only if $\rk(X) = r$.
\end{enumerate}
Note that the first three properties indicate that $\cl(X)$ does behave as a closure operator, while the last two depend on the rank function.
\end{proposition}

\begin{proof}
Property \ref{it:Xincl} is trivial.

\ref{it:cl_increasing}. For any $e \in \cl(X)$, we have
$$
    \rk(Y) + \rk(X) = \rk(Y) + \rk(X \cup \{e\}) \geq \rk(Y \cup \{e\}) + \rk(X),
$$
and hence $\rk(Y \cup \{e\}) = \rk(Y)$ and $e \in \cl(Y)$. Therefore, $\cl(X) \subseteq \cl(Y)$.

\ref{it:rk(cl)}. Denote the elements of $\cl(X) \backslash X$ as $e_1,\ldots,e_k$. We have
$$
    k \rk(X) = \sum_{i=1}^k \rk(X \cup \{e_i\}) \geq \rk(\cl(X)) + (k-1)\rk(X),
$$
and hence $\rk(\cl(X)) = \rk(X)$.

\ref{it:cl(cl)}. Let $e \in \cl(\cl(X))$, then by Property \ref{it:rk(cl)}
$$
     \rk(X) = \rk(\cl(X)) = \rk(\cl(X) \cup \{e\}) \geq \rk(X \cup \{e\}),
$$
and hence $e \in \cl(X)$.

\ref{it:cl=E}. If $\rk(X) = E$, then by Property \ref{it:rk(cl)} $r = \rk(\cl(X)) = \rk(X)$. Conversely, if $\rk(X) = r$, then for any $e \in E$, $r = \rk(X) = \rk(X \cup \{e\})$ and $e \in \cl(X)$.
\end{proof}

It is worth noticing that all operators satisfying Properties \ref{it:Xincl}, \ref{it:cl_increasing}, and \ref{it:cl(cl)} do not necessarily correspond to a rank function. An example is given below; in order to simplify notations, we shall remove brackets and commas when describing subsets of elements.

\begin{example}\label{ex:cl_not_rk}
Let $E = \{0,\ldots,5\}$ and $\mathcal{B}_{3,2} = \{012,345\}$. We define the closure operator $c$ as $c(X) = X$ for all $X \subseteq E$ except:
\begin{eqnarray*}
    c(X) &=& \{0\} \cup X \, \mbox{if} \, X \in \{34,134,234,35,135,235\}\\
    c(X) &=& E \, \mbox{if} \, X \mbox{contains a basis or}\, X \in \{0145,0245,1234,1235,1245\}
\end{eqnarray*}
%\begin{eqnarray*}
%    c(45) &=& 145\\
%    c(46) &=& 146\\
%    c(245) &=& 1245\\
%    c(345) &=& 1345\\
%    c(246) &=& 1246\\
%    c(346) &=& 1346\\
%    c(123) &=& E\\
%    c(456) &=& E\\
%    c(X) &=& E \quad \mbox{if} |X| \geq 4, X \notin \{1245,1246,1345,1346\}.
%\end{eqnarray*}
It can be shown that $c$ satisfies Properties \ref{it:Xincl}, \ref{it:cl_increasing}, and \ref{it:cl(cl)}. Furthermore, if $|X| = 3$, then $c(X) = E$ if and only if $X \in \mathcal{B}$. However, suppose $r$ is a rank function generating the closure operator $c$, then
$$
    4 \geq r(34) + r(35) = r(034) + r(035) \geq r(03) + r(0345) > 4,
$$
for $3 \notin c(0)$ and hence $r(03) > r(0) = 1$.
\end{example}

We refer to a set equal to its closure as a {\em flat}.

\begin{proposition}
We have the following properties.
\begin{enumerate}
    \item For any $X \subseteq E$, $\cl(X)$ is equal to the intersection of all flats containing $X$.

    \item The family of flats is closed under intersection.

    \item For any $X,Y \subseteq E$, $\cl(X \cup Y) = \cl(\cl(X) \cup \cl(Y))$.
\end{enumerate}
\end{proposition}

\begin{proof}
1. Let $F$ be a flat containing $X$ and let $e \in \cl(X)$, then
$$
    \rk(F) + \rk(X) = \rk(F) + \rk(X \cup \{e\}) \geq \rk(F \cup \{e\}) + \rk(X),
$$
and hence $e \in \cl(F) = F$.

2. For any flats $F,G$ we have $F \cap G \subseteq \cl(F \cap G)$. Moreover, we have
$$
    \cl(F \cap G) \subseteq \cl(F) = F,
$$
and similarly $\cl(F \cap G) \subseteq G$ and hence $\cl(F \cap G) \subseteq F \cap G$. Therefore, $F \cap G = \cl(F \cap G)$ is a flat.

3. Since $X \cup Y \subseteq \cl(X) \cup \cl(Y)$, we have $\cl(X \cup Y) \subseteq \cl(\cl(X) \cup \cl(Y))$. On the other hand, $\cl(X)$ and $\cl(Y)$ are both subsets of $\cl(X \cup Y)$ and hence $\cl(\cl(X) \cup \cl(Y)) \subseteq \cl(\cl(X \cup Y)) = \cl(X \cup Y)$.
\end{proof}

We can then define the {\em lattice of flats}, which is not necessarily semimodular (as the height function is not equal to the rank) but where the rank function is a semivaluation. Remark that although
$$
    \cl(X \cup Y) = \cl(\cl(X) \cup \cl(Y)),
$$
we do not necessarily have equality in the following:
$$
    \cl(X \cap Y) \subseteq \cl(X) \cap \cl(Y) = \cl(\cl(X) \cap \cl(Y)).
$$

We finish this section by noticing that the closure associated to the rank function of a given combinatorial representation satisfies
\begin{eqnarray*}
    \cl_{r_f}(X) &=& \{e \in E: r_f(X \cup \{e\}) = r_f(X)\}\\
    &=& \{e \in E : \bar{f}_{X \cup \{e\}} = \bar{f}_X\}\\
    &=& \{e \in E : \bar{f}_X \,\mbox{refines}\, \bar{f}_e\}.
\end{eqnarray*}

%\paragraph{Comment} The most natural rank function (and closure) for $\mathcal{M}(q,r)$ is the one generated by its canonical representation. Is there a relationship between the closure of elements in a family of bases and the closure of functions in any representation of that family?

\section{Representation by matrices with $2$ rows} \label{sec:2_rows}

In this section, we are interested in families which are ``nearly'' linearly representable matroids: those for which there is a representation with matrices with only two rows. By Proposition \ref{prop:cartesian}, if $(E,\mathcal{B})$ is the intersection of $d$ linear matroids, then it has a representation by matrices of size $d$. Proposition \ref{prop:B_23} gives a counterexample of the converse for $d=2$.

\begin{proposition} \label{prop:B_23}
The family $\mathcal{B}_{2,3}$ is not the intersection of two matroids, yet it has a representation by matrices with $2$ rows.
\end{proposition}

\begin{proof}
We first show that $\mathcal{B} := \mathcal{B}_{2,3}$ is not the intersection of two matroids. Recall that $E = \{0,\ldots,5\}$ and $\mathcal{B} = \{\{0,1\},\{2,3\},\{4,5\}\}$. Suppose on the contrary that $\mathcal{B} = \mathcal{B}_1 \cap \mathcal{B}_2$, where $\mathcal{B}_i$ is a matroid on $E$ for $i \in \{1,2\}$. Since $r=2$, we view these as graphs, and we say that two vertices are adjacent if they form a basis. By the exchange axiom, for each vertex $e$ and each basis $b \in \mathcal{B}$ not containing $e$, $e$ is  adjacent in $\mathcal{B}_1$ and $\mathcal{B}_2$ to one or two vertices of $b$. Since such an edge does not appear in $\mathcal{B}_1 \cap \mathcal{B}_2$, we conclude that $e$ is adjacent in $\mathcal{B}_1$ to exactly one vertex of $b$ and is adjacent in $\mathcal{B}_2$ to the other vertex of $b$. Without loss of generality, let $0$ be adjacent to $2$ and $4$ in $\mathcal{B}_1$: $\{0,2\}, \{0,4\} \in \mathcal{B}_1$, then $\{1,2\}, \{1,4\} \notin \mathcal{B}_1$ by applying the conclusion above to $b = \{0,1\}$ and $e = 2$ and $e=4$, respectively. However $\{0,2\}, \{0,4\} \notin \mathcal{B}_2$ show that $0,2,4$ are parallel in $\mathcal{B}_2$ and hence $\{2,4\} \notin \mathcal{B}_2$. Therefore, $\{2,4\} \in \mathcal{B}_1$ while $\{1,2\}, \{1,4\} \notin \mathcal{B}_1$, and the exchange axiom is violated: this is the desired contradiction.

We now give a representation of $(E,\mathcal{B})$ using matrices with $2$ rows. Let $p$ be any prime power, and consider $A = \mathrm{GF}(p)^2$, then we can express $(x,y) \in A^2$ as $(x_1,x_2,y_1,y_2) \in \gf(p)^4$. The functions representing $(E,\mathcal{B})$ can be expressed as $f_i(x_1,x_2,y_1,y_2) = {\bf F}_i (x_1,x_2,y_1,y_2)^T$, where ${\bf F}_i \in \gf(p)^{2 \times 4}$ are given by
\begin{eqnarray*}
    {\bf F}_0 &=& \begin{pmatrix}
    1 & 0 & 0 & 0\\
    0 & 1 & 0 & 0
    \end{pmatrix},\\
    {\bf F}_1 &=& \begin{pmatrix}
    0 & 0 & 1 & 0\\
    0 & 0 & 0 & 1
    \end{pmatrix},\\
    {\bf F}_2 &=& \begin{pmatrix}
    1 & 0 & 0 & 0\\
    0 & 0 & 1 & 0
    \end{pmatrix},\\
    {\bf F}_3 &=& \begin{pmatrix}
    0 & 1 & 0 & 0\\
    0 & 0 & 0 & 1
    \end{pmatrix},\\
    {\bf F}_4 &=& \begin{pmatrix}
    1 & 0 & 0 & 0\\
    0 & 0 & 0 & 1
    \end{pmatrix},\\
    {\bf F}_5 &=& \begin{pmatrix}
    0 & 1 & 0 & 0\\
    0 & 0 & 1 & 0
    \end{pmatrix}.
\end{eqnarray*}
\end{proof}

We now prove that there exist families of bases which do not have a representation using matrices with $2$ rows. The proof is based on the Ingleton inequality~\cite{Ing71}: for any four subsets $X_1,\ldots,X_4$ of a family representable using matrices, we have
\begin{eqnarray}
    \nonumber
    &&\rk(X_1) + \rk(X_2) + \rk(X_1 \cup X_2 \cup X_3) + \rk(X_1 \cup X_2 \cup X_4) + \rk(X_3 \cup X_4)\\
    \label{eq:ingleton}
    &&\leq \rk(X_1 \cup X_2) + \rk(X_1 \cup X_3) + \rk(X_1 \cup X_4) + \rk(X_2 \cup X_3) + \rk(X_2 \cup X_4).
\end{eqnarray}

\begin{proposition} \label{prop:not_2_rows}
Let $E = \{1,\ldots,7\}$ and $\mathcal{B} = \{16, 27, 34, 45, 53\}$ as represented in Figure \ref{fig:graph}. Then $(E,\mathcal{B})$ does not have a representation using matrices with $2$ rows.
\end{proposition}

\begin{figure}
\begin{center}
\setlength{\unitlength}{1cm}
\begin{picture}(7,3)
%\begin{tikzpicture}
%    \tikzstyle{every node}=[draw,shape=circle];
%
%    \node (1) at (2,2) {1};
\put(2,2){\circle{1}}
\put(1.9,1.85){$1$}
%    \node (2) at (2,0) {2};
\put(2,0){\circle{1}}
\put(1.9,-0.15){$2$}
%    \node (3) at (4,2) {3};
\put(4,2){\circle{1}}
\put(3.9,1.85){$3$}
%    \node (4) at (4,0) {4};
\put(4,0){\circle{1}}
\put(3.9,-0.15){$4$}
%    \node (5) at (5,1) {5};
\put(6,1){\circle{1}}
\put(5.9,0.85){$5$}
%    \node (6) at (0,2) {6};
\put(0,2){\circle{1}}
\put(-0.1,1.85){$6$}
%    \node (7) at (0,0) {7};
\put(0,0){\circle{1}}
\put(-0.1,-0.15){$7$}
%
%    \draw (1) -- (6);
\put(0.5,2){\line(1,0){1}}
%    \draw (2) -- (7);
\put(0.5,0){\line(1,0){1}}
%    \draw (3) -- (4);
\put(4,0.5){\line(0,1){1}}
%    \draw (4) -- (5);
\put(4.448,0.224){\line(2,1){1.116}}
%    \draw (5) -- (3);
\put(4.448,1.776){\line(2,-1){1.116}}
%\end{tikzpicture}
\end{picture}
\end{center}
\caption{A graph without representation by matrices with $2$ rows.} \label{fig:graph}
\end{figure}
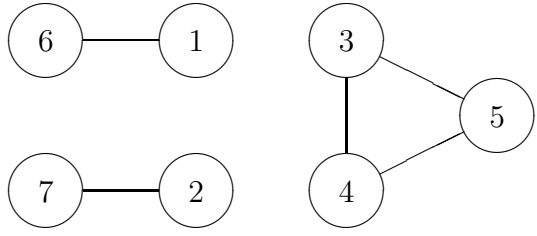

\begin{proof}
Any rank function $\rk(X)$ generated by a representation using matrices with $2$ rows takes half-integer values. However, let $\rk(X)$ be a rank function with half-integer values and let us prove that it violates the Ingleton inequality in (\ref{eq:ingleton}). First, we have $\rk(12), \rk(17) \leq 1.5$ as neither are bases and
$$
    \rk(12) + \rk(17) \geq \rk(1) + \rk(127) = 3,
$$
which implies that $\rk(12) = \rk(17) = 1.5$.

Second, submodularity implies
$$
    \rk(123) + \rk(124) \geq \rk(12) + \rk(1234) = 3.5,
$$
and hence $\rk(123) = 2$ or $\rk(124) = 2$; without loss, say $\rk(123) = 2$. By symmetry, we also obtain that $\rk(124) = 2$ or $\rk(125) = 2$; say $\rk(124) = 2$.

Denoting $X_i = \{i\}$ for $i = 1,\ldots,4$, the Ingleton inequality is violated:
\begin{eqnarray*}
    \rk(1) + \rk(2) + \rk(123) + \rk(124) + \rk(34) = 8 &>\\
    \rk(12) + \rk(13) + \rk(14) + \rk(23) + \rk(24) = 7.5.
\end{eqnarray*}
\end{proof}

By the same technique, one can also show that the following family of bases is not representable using matrices with $2$ rows:
$$
    E = \{0,1,\ldots,8\}, \mathcal{B} = \{012,034,056,078,057,135,137\}.
$$
The Ingleton inequality is violated by
$$
    X_1 = 01, X_2 = 03, X_3 = 05, X_4 = 07.
$$

However, Proposition \ref{prop:Ingleton_3} shows that the Ingleton inequality cannot rule out representations by matrices with $3$ rows.

\begin{proposition} \label{prop:Ingleton_3}
The rank function defined by
\begin{equation}
    \rk(X) = \begin{cases}
    |X| & \mathrm{if}\, |X| \leq r-1 \,\mathrm{or}\, X \in \mathcal{B}\\
    r- \frac{1}{3} & \mathrm{if}\, |X| = r, X \notin \mathcal{B}\\
    r & \mathrm{if}\, |X| \geq r+1
    \end{cases}
\end{equation}
satisfies the Ingleton inequality.
\end{proposition}

\begin{proof}
The fact that this is a rank function was proved in Theorem \ref{th:half_integers}. We denote the left and right hand sides of the Ingleton inequality in (\ref{eq:ingleton}) as $L$ and $R$, respectively. The proof goes by considering cases based on the terms in $R$. It is important to note the symmetric roles of $X_1$ and $X_2$ on one hand and $X_3$ and $X_4$ on the other hand in Ingleton's inequality. In order to illustrate our calculations, square brackets indicate where we use the submodular inequality. %First, we rule out the cases where a term in $R$ is equal to $r$ by considering looser conditions.

\paragraph{Case I:} $\rk(X_1 \cup X_2) = \rk(X_1 \cup X_2 \cup X_3)$. We then have
\begin{eqnarray*}
    R &=& \rk(X_1 \cup X_2 \cup X_3) + [\rk(X_1 \cup X_3) + \rk(X_1 \cup X_4)] + [\rk(X_2 \cup X_3) + \rk(X_2 \cup X_4)]\\
    &\geq& \rk(X_1 \cup X_2 \cup X_3) + \rk(X_1) + \rk(X_2) + [\rk(X_1 \cup X_3\cup X_4) + \rk(X_2 \cup X_3 \cup X_4)]\\
%    &\geq& \rk(X_1 \cup X_2 \cup X_3) + \rk(X_1) + \rk(X_2) + \rk(X_3\cup X_4) + \rk(X_1 \cup X_2 \cup X_4) = L
    &\geq& L.
\end{eqnarray*}
By symmetry, we also rule out the case where $\rk(X_1 \cup X_2) = \rk(X_1 \cup X_2 \cup X_4)$.

\paragraph{Case II:} $\rk(X_1 \cup X_3) = \rk(X_1 \cup X_2 \cup X_3)$. We then have
\begin{eqnarray*}
    R &=& \rk(X_1 \cup X_2 \cup X_3) + [\rk(X_1 \cup X_2) + \rk(X_1 \cup X_4)] + [\rk(X_2 \cup X_3) + \rk(X_2 \cup X_4)]\\
%    &\geq& \rk(X_1 \cup X_2 \cup X_3) + \rk(X_1) + \rk(X_2) + \rk(X_1 \cup X_2 \cup X_4) + \rk(X_3 \cup X_4) = L
    &\geq& L.
\end{eqnarray*}
By symmetry, we also rule out the cases where $\rk(X_1 \cup X_4) = \rk(X_1 \cup X_2 \cup X_4)$, $\rk(X_2 \cup X_3) = \rk(X_1 \cup X_2 \cup X_3)$, or $\rk(X_2 \cup X_4) = \rk(X_1 \cup X_2 \cup X_4)$. Remark that the case where some term in $R$ is equal to $r$ is contained in Case I or Case II, therefore all ranks are less than $r$ in the right hand side in the next Cases.

\paragraph{Case III:} All the terms in $R$ have cardinality at most $r-1$, and hence their rank is equal to their cardinality. It is clear that the cardinality function satisfies the Ingleton inequality, for $|X| = \dim(V_X)$ for all $X \subseteq E$, where $V_X$ is the subspace generated by the unit vectors $\{e_i: i \in X\}$.
%This case is settled by explicitly writing out both sides of the equation; we obtain $R = \bar{R}$, where
%\begin{eqnarray*}
%    \bar{R} &:=& |X_1| + |X_2| + |X_1 \cup X_2 \cup X_3| + |X_1 \cup X_2 \cup X_4| + |X_3 \cup X_4|\\
%    &=& \bar{L} + |(X_1 \cap X_2) \backslash (X_3 \cup X_4)| + |(X_1 \cap X_3 \cap X_4) \backslash X_2|\\
%     &&+ |(X_2 \cap X_3 \cap X_4) \backslash X_1| + |(X_3 \cap X_4) \backslash (X_1 \cup X_2)|,\\
%     \bar{L} &:=& |X_1 \cup X_2| + |X_1 \cup X_3| + |X_1 \cup X_4| + |X_2 \cup X_3| + |X_2 \cup X_4| \geq L.
%\end{eqnarray*}

\paragraph{Case IV:} Some terms in $R$ have rank $r-\frac{1}{3}$, say $1 \leq k \leq 5$ of them. If $k \leq 3$, then $\rk(X_1 \cup X_2 \cup X_3) = r \leq |X_1 \cup X_2 \cup X_3| - 1$ or $\rk(X_1 \cup X_2 \cup X_4) = r \leq |X_1 \cup X_2 \cup X_4| - 1$. This holds since one of the two terms in $L$ must have a greater rank than one of the $k$ terms in $R$ with rank $r-\frac{1}{3}$. Then the term in $L$ properly contains the according one in $R$, and its cardinality is at least $r+1$. Let us denote
\begin{eqnarray*}
    \bar{R} &:=& |X_1| + |X_2| + |X_1 \cup X_2 \cup X_3| + |X_1 \cup X_2 \cup X_4| + |X_3 \cup X_4|,\\
%    &=& \bar{L} + |(X_1 \cap X_2) \backslash (X_3 \cup X_4)| + |(X_1 \cap X_3 \cap X_4) \backslash X_2|\\
%     &&+ |(X_2 \cap X_3 \cap X_4) \backslash X_1| + |(X_3 \cap X_4) \backslash (X_1 \cup X_2)|,\\
     \bar{L} &:=& |X_1 \cup X_2| + |X_1 \cup X_3| + |X_1 \cup X_4| + |X_2 \cup X_3| + |X_2 \cup X_4|.
\end{eqnarray*}
Thus
$$
    R = \bar{R} - \frac{k}{3} \geq \bar{R} - 1 \geq \bar{L} - 1 \geq L.
$$
If $k \in \{4,5\}$, then both $X_1 \cup X_2 \cup X_3$ and $X_1 \cup X_2 \cup X_4$ have rank $r$, and
$$
    R = \bar{R} - \frac{k}{3} \geq \bar{R} - 2 \geq \bar{L} - 2 \geq L.
$$
\end{proof}

\section{Conclusion} \label{sec:conclusion}

The representation of graphs in Section \ref{sec:graphs}) yields a couple of open questions.
\begin{itemize}
\item Given a graph, what is the largest alphabet over which it is not
representable?
\item Does Theorem \ref{th:graph} hold for $r$-uniform hypergraphs with $r>2$?
\end{itemize}

Also, after generalising  loops and parallel elements in Section \ref{sec:given_alphabet} and rank functions and closure operators in Section \ref{sec:rk_cl}, one wonders if more concepts from matroid theory could be generalised in the framework of combinatorial representations.

The relation with information theory via rank functions and especially the submodular inequality needs to be further investigated. Indeed, a wealth of non-Shannon inequalities have been discovered recently~\cite{ZY97,MMRV02,Zha03,DFZ06a}, see~\cite{Cha11} for a survey on this matter. However, it seems rather unclear how much more information can be drawn from all these new inequalities and how hard they are to manipulate. Similarly, non-Ingleton inequalities have been discovered for the dimension of intersections of linear subspaces~\cite{Kin11}. Once again, what conclusions can we draw from these inequalities?

% Generated by IEEEtran.bst, version: 1.13 (2008/09/30)

\end{document}